\documentclass[11pt]{article}

\usepackage[english]{babel}
\usepackage[utf8x]{inputenc}
\usepackage[T1]{fontenc}
\usepackage[letterpaper,top=1in,bottom=1in,left=1in,right=1in,marginparwidth=1.75cm]{geometry}

\usepackage{amsmath,amsfonts,amsthm,amssymb,mathrsfs,dsfont} 
\usepackage{mathtools}
\usepackage{graphicx}
\usepackage[colorinlistoftodos]{todonotes}
\usepackage[colorlinks=true, allcolors=blue]{hyperref}
\usepackage[capitalize,nameinlink]{cleveref}
\usepackage{verbatim}
\usepackage{enumerate}

\newcommand{\R}{\mathbb{R}}

\global\long\def\R{\mathbb{R}}

\global\long\def\S{\mathbb{S}}

\global\long\def\E{\mathbb{E}}

\global\long\def\S{\mathcal{S}}

\global\long\def\Pr{\text{Pr}}

\global\long\def\Unif{\text{Unif}}

\newcommand{\vnote}[1]{\textcolor{red}{\small {\textbf{(Vishesh: }#1\textbf{) }}}}

\newtheorem{theorem}{Theorem}[section]
\newtheorem*{namedtheorem}{\theoremname}
\newcommand{\theoremname}{testing}

\newtheorem{lemma}[theorem]{Lemma}
\newtheorem{problem}[theorem]{Problem}

\newtheorem{proposition}[theorem]{Proposition}

\newtheorem{corollary}[theorem]{Corollary}

\newtheorem*{question*}{Question}

\theoremstyle{definition}
\newtheorem{definition}[theorem]{Definition}

\newtheorem{remark}[theorem]{Remark}

\theoremstyle{plain}

\usepackage{algpseudocode,algorithm,algorithmicx}

\title{1-factorizations of pseudorandom 
graphs}
\author{Asaf Ferber \thanks{Massachusetts Institute of Technology. Department of Mathematics. Email: {\tt ferbera@mit.edu}. Research is partially supported by an NSF grant 6935855.} \and Vishesh Jain\thanks{Massachusetts Institute of Technology. Department of Mathematics. Email: {\tt visheshj@mit.edu}} }
\date{}

\begin{document}
\maketitle

\begin{abstract} A $1$-factorization of a graph $G$ is a collection of edge-disjoint perfect matchings whose union is $E(G)$. A trivial necessary condition for $G$ to admit a $1$-factorization is that $|V(G)|$ is even and $G$ is regular; the converse is easily seen to be false. In this paper, we consider the problem of finding $1$-factorizations of regular, pseudorandom graphs. Specifically, we prove that an $(n,d,\lambda)$-graph $G$ (that is, a $d$-regular graph on $n$ vertices whose second largest eigenvalue in absolute value is at most $\lambda$) admits a $1$-factorization provided that $n$ is even, $C_0\leq d\leq n-1$ (where $C_0$ is a universal constant), and $\lambda\leq d^{0.9}$. In particular, since (as is well known) a typical random $d$-regular graph $G_{n,d}$ is such a graph, we obtain the existence of a $1$-factorization in a typical $G_{n,d}$ for all $C_0\leq d\leq n-1$, thereby extending to all possible values of $d$ results obtained by Janson, and independently by Molloy, Robalewska, Robinson, and Wormald for fixed $d$. Moreover, we also obtain a lower bound for the number of distinct $1$-factorizations of such graphs $G$ which is off by a factor of $2$ in the base of the exponent from the known upper bound. This lower bound is better by a factor of $2^{nd/2}$ than the previously best known lower bounds, even in the simplest case where $G$ is the complete graph. Our proofs are probabilistic and can be easily turned into polynomial time (randomized) algorithms.  

\end{abstract}

\section{Introduction}

The \emph{chromatic index} of a graph $G$, denoted by $\chi'(G)$, is the minimum number of colors 
with which it is possible to color the edges of $G$ in a way such that every color class consists of a matching (that is, no two edges of the same color share a vertex). This parameter is one of the most fundamental and widely studied 
parameters in graph theory and combinatorial optimization and in particular, is related to optimal scheduling and resource allocation problems and round-robin tournaments (see, e.g., \cite{GDP}, \cite{Lynch}, \cite{MR}). 

A trivial lower bound on $\chi'(G)$ 
is $\chi'(G)\geq \Delta(G)$, where $\Delta(G)$ denotes the maximum degree of $G$. Indeed, consider any vertex with maximum degree, and observe that all edges incident to this vertex must have distinct colors. 
Perhaps surprisingly, a classical theorem of Vizing \cite{Vizing} from the 1960s shows that $\Delta + 1$ colors are always sufficient, and therefore, $\chi'(G)\in \{\Delta(G),\Delta(G)+1\}$ holds for all graphs. In particular, this shows that one can partition all graphs into two classes: \emph{Class 1} consists of all graphs $G$ for which $\chi'(G)=\Delta(G)$, and \emph{Class 2} consists of all graphs $G$ for which $\chi'(G)=\Delta(G) + 1$. Moreover, the strategy in Vizing's original proof can be used to obtain a polynomial time algorithm to edge color any graph $G$ with $\Delta(G) + 1$ colors (\cite{MG}).
However, Holyer \cite{Hoy} showed that it is actually NP-hard to decide whether a given graph $G$ is in Class 1 or 2. In fact, Leven and Galil \cite{Leven} showed that this is true even if we restrict ourselves to graphs with all the degrees being the same (that is, to \emph{regular graphs}).

Note that for $d$-regular graphs $G$ (that is, graphs with all their degrees equal to $d$) on an even number of vertices, the statement `$G$ is of Class 1' is equivalent to the statement that $G$ contains $d$ edge-disjoint \emph{perfect matchings} (also known as \emph{$1$-factors}). A graph whose edge set decomposes as a disjoint union of perfect matchings is said to admit a \emph{$1$-factorization}.
Note that if $G$ is a $d$-regular \emph{bipartite} graph, then a straightforward application of Hall's marriage theorem immediately shows that $G$ is of Class 1. Unfortunately, the problem is much harder for non-bipartite graphs, and it is already very interesting to find (efficiently verifiable) sufficient conditions 
which ensure that $\chi'(G)=\Delta(G)$. This problem is the main focus of our paper.

\subsection{Regular expanders are of Class 1}

Our main result shows that $d$-regular graphs on an even number of vertices which are `sufficiently good' spectral expanders, are of Class 1. Before stating our result precisely, we need to introduce some notation and definitions. Given a $d$-regular graph $G$ on $n$ vertices, let $A(G)$ be its adjacency matrix (that is, $A(G)$ is an $n\times n$, $0/1$-valued matrix, with $A(G)_{ij}=1$ if and only if $ij\in E(G)$). Clearly, $A(G)\cdot \textbf{1}=d\textbf{1}$, where $\textbf{1}\in \mathbb{R}^n$ is the vector with all entries  equal to $1$, and therefore, $d$ is an eigenvalue of $A(G)$. In fact, as can be easily proven, $d$ is the eigenvalue of $A(G)$ with largest absolute value. Moreover, since $A(G)$ is a symmetric, real-valued matrix, it has $n$ real eigenvalues (counted with multiplicities). Let 
$$d:=\lambda_1\geq \lambda_2\geq \ldots \lambda_n \geq -d$$ denote the eigenvalues of $A(G)$, and let $\lambda(G):=\max\{|\lambda_2|,|\lambda_n|\}$. With this notation, we say that $G$ is an $(n,d,\lambda)$-graph if $G$ is a $d$-regular graph on $n$ vertices with $\lambda(G)\leq \lambda$. In recent decades, the study of $(n,d,\lambda)$ graphs, also known as `spectral expanders', has attracted considerable attention in mathematics and theoretical computer science. An example which is relevant to our problem is that of finding a perfect matching in $(n,d,\lambda)$-graphs. Extending a result of Krivelevich and Sudakov \cite{KS}, Ciob\v{a}, Gregory and Haemers \cite{CGH} proved accurate spectral conditions for an $(n,d,\lambda)$-graph to contain a perfect matching. For much more on these graphs and their many applications, we refer the reader to the surveys of Hoory, Linial and Wigderson \cite{HLW}, Krivelevich and Sudakov \cite{KS}, and to the book of Brouwer and Haemers \cite{BH}. We are now ready to state our main result.   

\begin{theorem}
\label{main:pseudorandom}
For every $\varepsilon>0$ there exist $d_0,n_0 \in \mathbb{N}$ such that for all even integers $n\geq n_0$ and for all $d \geq d_0$ the following holds. Suppose that $G$ is an $(n,d,\lambda)$-graph with $\lambda\leq d^{1-\varepsilon}$. Then, $\chi'(G)=d$. 
\end{theorem}
\begin{remark}
It seems plausible that with a more careful analysis of our proof, one can improve our bound to $\lambda\leq d/poly(\log{d})$. Since we believe that the actual bound should be much stronger, we did not see any reason to optimize our bound at the expense of making the paper more technical.   
\end{remark}

In particular, since the eigenvalues of a matrix can be computed in polynomial time, \cref{main:pseudorandom} provides a polynomial time checkable sufficient condition for a graph to be of Class 1. Moreover, our proof gives a probabilistic polynomial time algorithm to actually find an edge coloring of such a $G$ with $d$ colors. Our result can be viewed as implying that `sufficiently good' spectral expanders are easy instances for the NP-complete problem of determining the chromatic index of regular graphs. It is interesting (although, perhaps a bit unrelated) to note that in recent work, Arora et al. \cite{Arora} showed that constraint graphs which are reasonably good spectral expanders are easy for the conjecturally NP-complete Unique Games problem as well. 


\subsection{Almost all $d$-regular graphs are of Class 1}

The phrase `almost all $d$-regular graphs' usually splits into two cases: `dense' graphs and random graphs. Let us start with the former. \\

{\bf Dense graphs: }It is well known (and quite simple to prove) that every $d$-regular graph $G$ on $n$ vertices, with $d\geq 2\lceil n/4\rceil-1$ has a perfect matching (assuming, of course, that $n$ is even). Moreover, for every $d\leq 2\lceil n/4\rceil-2$, it is easily seen that there exist $d$-regular graphs on an even number of vertices that do not contain even one perfect matching.
In a (relatively) recent breakthrough, Csaba, K\"uhn, Lo, Osthus, and Treglown \cite{CKLOT} proved a longstanding conjecture of Dirac from the 1950s, and showed that the above minimum degree condition is tight, not just for containing a single perfect matching, but also for admitting a $1$-factorization. 

\begin{theorem} [Theorem 1.1.1 in \cite{CKLOT}]
\label{theorem: long}
Let $n$ be a sufficiently large even integer, and let $d\geq 2\lceil n/4\rceil-1$. Then, every $d$-regular graph $G$ on $n$ vertices contains a $1$-factorization.
\end{theorem}

Hence, every sufficiently `dense' regular graph is of Class 1. It is worth mentioning that they actually proved a much more general statement about finding edge-disjoint \emph{Hamilton cycles}, from which the above theorem follows as a corollary.\\ 

\textbf{Random graphs: }As noted above, one cannot obtain a statement like \cref{theorem: long} for smaller values of $d$ since the graph might not even have a single perfect matching. 
Therefore, a natural candidate to consider for such values of $d$ is the random $d$-regular graph, denoted by $G_{n,d}$, which is simply 
a random variable that outputs a $d$-regular graph on $n$ vertices, chosen uniformly at random from among all such graphs. The study of this random graph model has received much interest in recent years. Unlike the traditional binomial random graph $G_{n,p}$ (where each edge of the complete graph is included independently, with probability $p$), the uniform regular model has many dependencies, and is therefore much harder to work with. For a detailed discussion of this model, along with many results and open problems, we refer the reader to the survey of Wormald \cite{Wormald}. 

Working with this model, Janson \cite{Jans}, and independently, Molloy, Robalewska, Robinson, and Wormald \cite{M}, proved that a typical $G_{n,d}$ admits a $1$-factorization for all fixed $d\geq 3$, where $n$ is a sufficiently large (depending on $d$) even integer. Later, Kim and Wormald \cite{KW} gave a randomized algorithm to decompose a typical $G_{n,d}$ into $\lfloor \frac{d}{2}\rfloor$ edge-disjoint \emph{Hamilton cycles} (and an additional perfect matching if $d$ is odd) under the same assumption that $d \geq 3$ is fixed, and $n$ is a sufficiently large (depending on $d$) even integer. The main problem with handling values of $d$ which grow with $n$ is that the so-called `configuration model' (see \cite{Bol} for more details) does not help us in this regime.

Here, as an almost immediate corollary of \cref{main:pseudorandom}, we deduce the following, which together with the results of \cite{Jans} and \cite{M} shows that a typical $G_{n,d}$ on a sufficiently large even number of vertices admits a $1$-factorization for all $3 \leq d \leq n-1$. 

\begin{corollary}
\label{main}
There exists a universal constant $d_0\in \mathbb{N}$ such that for all $d_0\leq d\leq n-1$, a random $d$-regular graph $G_{n,d}$ admits a $1$-factorization asymptotically almost surely (a.a.s.). 
\end{corollary}
\begin{remark} 
By asymptotically almost surely, we mean with probability going to $1$ as $n$ goes to infinity (through even integers). Since a $1$-factorization can never exist when $n$ is odd, we will henceforth always assume that $n$ is even, even if we do not explicitly state it.  
\end{remark}

To deduce \cref{main} from \cref{main:pseudorandom}, it suffices to show that we have (say) $\lambda(G_{n,d}) = O(d^{0.9})$ a.a.s. In fact, the considerably stronger (and optimal up to the choice of constant in the big-oh) bound that $\lambda(G_{n,d}) = O(\sqrt{d})$ a.a.s. is known. For $d = o(\sqrt{n})$, this is due to Broder, Frieze, Suen and Upfal \cite{BFSU}. This result was extended to the range $d=O(n^{2/3})$ by Cook, Goldstein, and Johnson \cite{CGJ} and to all values of $d$ by Tikhomirov and Youssef \cite{T}. We emphasize that the condition on $\lambda$ we require is significantly weaker and can possibly be deduced from much simpler arguments than the ones in the references above.\\

It is also worth mentioning that very recently, Haxell, Krivelevich and Kronenberg \cite{HKK} studied a related problem in a random multigraph setting; it is interesting to check whether our techniques can be applied there as well.

\subsection{Counting $1$-factorizations}
Once the existence of $1$-factorizations in a family of graphs has been established, it is natural to ask for the number of \emph{distinct} $1$-factorizations that any member of such a family admits. Having a `good' approximation to the number of $1$-factorizations can shed some light on, for example, properties of a `randomly selected' $1$-factorization. We remark that the case of counting the number of $1$-factors (perfect matchings), even for bipartite graphs, has been the subject of fundamental works over the years, both in combinatorics (e.g., \cite{Bregman}, \cite{Egorychev}, \cite{Falikman}, \cite{Schrijver}), as well as in theoretical computer science (e.g., \cite{Valiant}, \cite{JSV}), and had led to many interesting results such as both closed-form as well as computational approximation results for the permanent of $0/1$ matrices.  

As far as the question of counting the number of $1$-factorizations is concerned, much less is known. Note that for $d$-regular bipartite graphs, one can use estimates on the permanent of the adjacency matrix of $G$ to obtain quite tight results. But quite embarrassingly, for non-bipartite graphs (even for the complete graph!) the number of $1$-factorizations in unknown. The best known upper bound for the number of $1$-factorizations in the complete graph is due to Linial and Luria \cite{LL}, who showed that it is upper bounded by
$$\left((1+o(1))\frac{n}{e^2}\right)^{n^2/2}.$$ 
Moreover, by following their argument verbatim, one can easily 
show that the number of 1-factorizations of any $d$-regular graph is at most
$$\left((1+o(1))\frac{d}{e^2}\right)^{dn/2}.$$ 
On the other hand, the previously best known lower bound for the number of $1$-factorizations of the complete graph (\cite{Cameron}, \cite{Zinoviev}) is only 
$$\left((1+o(1))\frac{n}{4e^2}\right)^{n^{2}/2},$$
which is off by a factor of $4^{n^2/2}$ from the upper bound.

An immediate advantage of our proof is that it gives a lower bound on the number of $1$-factorizations which is better than the one above by a factor of $2$ in the base of the exponent, not just for the complete graph, but for all sufficiently good regular spectral expanders with degree greater than some large constant. More precisely, we will show the following (see also the third bullet in \cref{section-concluding-rmks}) 

\begin{theorem}
\label{main:counting} For any $\epsilon >0$, there exist $D=D(\epsilon), N=N(\epsilon) \in \mathbb{N}$ such that for all even integers $n\geq N(\epsilon)$ and for all $d\geq D(\epsilon)$, the number of $1$-factorizations in any $(n,d,\lambda)$-graph with $\lambda \leq d^{0.9}$ is at least 
$$\left((1-\epsilon)\frac{d}{2e^2}\right)^{dn/2}.$$
\end{theorem}
\begin{remark} As discussed before, this immediately implies that for all $d\geq D(\epsilon)$, the number of $1$-factorizations of $G_{n,d}$ is a.a.s. 
$$\left((1-\epsilon)\frac{d}{2e^2}\right)^{dn/2}.$$
\end{remark}

\subsection{Outline of the proof}
It is well known, and easily deduced from Hall's theorem, that any regular bipartite graph admits a $1$-factorization (\cref{Hall-application-1fact-reg-bipartite}). Therefore, if we had a decomposition $E(G)=E(H'_1)\cup\dots E(H'_t)\cup E(\mathcal{F})$, where $H'_1,\dots H'_t$ are regular balanced bipartite graphs, and $\mathcal{F}$ is a $1$-factorization of the regular graph $G\setminus\bigcup_{i=1}^{t}H'_i$, we would be done. Our proof of \cref{main:pseudorandom} will obtain such a decomposition constructively. 

As shown in \cref{existence-good-graph}, one can find a collection of edge disjoint, regular bipartite graphs $H_1,\dots,H_t$, where $t\ll d$ and each $H_i$ is $r_i$ regular, with $r_i\approx d/t$, which covers `almost' all of $G$. In particular, one can find an `almost' $1$-factorization of $G$. However, it is not clear how to complete an arbitrary such `almost' $1$-factorization to an actual $1$-factorization of $G$. To circumvent this difficulty, we will adopt the following strategy. Note that $G':=G\setminus\bigcup_{i=1}^{t}H_i$ is a $k$-regular graph with $k\ll d$, and we can further force $k$ to be even (for instance, by removing a perfect matching from $H_1$). Therefore, by Petersen's $2$-factor theorem (\cref{2-factor theorem}), we  easily obtain a decomposition $E(G')=E(G'_1)\cup\dots E(G'_t)$, where each $G'_i$ is approximately $k/t$ regular. The key ingredient of our proof (\cref{proposition-completion}) then shows that the $H_i$'s can initially be chosen in such a way that each $R_i:=H_i\cup G'_i$ can be edge decomposed into a regular balanced bipartite graph, and a relatively small number of $1$-factors.

The basic idea in this step is quite simple. Observe that while the regular graph $R_i$ is not bipartite, it is `close' to being one, in the sense that most of its edges come from the regular balanced bipartite graph $H_i=(A_i\cup B_i, E_i)$. Let $R_i[A_i]$ denote the graph induced by $R_i$ on the vertex set $A_i$, and similarly for $B_i$, and note that the number of edges $e(R_i[A_i])=e(R_i[B_i])$. We will show that $H_i$ can be taken to have a certain `goodness' property (\cref{defn-good-graph}) which, along with the sparsity of $G'_i$, enables one to perform the following process to `absorb' the edges in $R_i[A_i]$ and $R_i[B_i]$: decompose $R_i[A_i]$ and $R_i[B_i]$ into the same number of matchings, with corresponding matchings of equal size, and complete each such pair of matchings to a perfect matching of $R_i$. After removing all the perfect matchings of $R_i$ obtained in this manner, we are clearly left with a regular balanced bipartite graph, as desired.     


Finally, for the lower bound on the number of $1$-factorizations, we show that there are many ways of performing such an edge decomposition $E(G)=E(H'_1)\cup\dots\cup E(H'_t)\cup E(\mathcal{F})$ (\cref{rmk:counting}), and there are many $1$-factorizations corresponding to each choice of edge decomposition (\cref{rmk:counting-completion}). 
\section{Tools and auxiliary results}
In this section we have collected a number of tools and auxiliary results to be used in proving our main theorem.

\subsection{Probabilistic tools}
	
Throughout the paper, we will make extensive use of the following well-known concentration inequality due to Hoeffding (\cite{Hoeff}).
\begin{lemma}[Hoeffding's inequality]
\label{Hoeffding}
Let $X_{1},\dots,X_{n}$ be independent random variables such that
$a_{i}\leq X_{i}\leq b_{i}$ with probability one. If $S_{n}=\sum_{i=1}^{n}X_{i}$,
then for all $t>0$,
\[
\Pr\left(S_{n}-\E[S_{n}]\geq t\right)\leq\exp\left(-\frac{2t^{2}}{\sum_{i=1}^{n}(b_{i}-a_{i})^{2}}\right)
\]
and 
\[
\Pr\left(S_{n}-\E[S_{n}]\leq-t\right)\leq\exp\left(-\frac{2t^{2}}{\sum_{i=1}^{n}(b_{i}-a_{i})^{2}}\right).
\]
\end{lemma}
Sometimes, we will find it more convenient to use the following bound on the
	upper and lower tails of the Binomial distribution due to Chernoff
	(see, e.g., Appendix A in \cite{AlonSpencer}).

		\begin{lemma}[Chernoff's inequality]
		Let $X \sim Bin(n, p)$ and let
		${\mathbb E}(X) = \mu$. Then
			\begin{itemize}
				\item
				$\Pr[X < (1 - a)\mu ] < e^{-a^2\mu /2}$
 				for every $a > 0$;
				\item $\Pr[X > (1 + a)\mu ] <
				e^{-a^2\mu /3}$ for every $0 < a < 3/2$.
			\end{itemize}
		\end{lemma}
\noindent
	\begin{remark}\label{rem:hyper} These bounds also hold when $X$ is
	hypergeometrically distributed with
	mean $\mu $.
\end{remark}
Before introducing the next tool to be used, we need the following
definition.

\begin{definition}
  Let $(A_i)_{i=1}^n$ be a collection of events in some probability
space. A graph $D$ on the vertex set $[n]$ is called a
\emph{dependency graph} for $(A_i)_i$ if $A_i$ is mutually
independent of all the events $\{A_j: ij\notin E(D)\}$.
\end{definition}

The following is the so called Lov\'asz Local Lemma, in its symmetric version (see, e.g., \cite{AlonSpencer}).

\begin{lemma}[Local Lemma]\label{LLL}
Let $(A_i)_{i=1}^n$ be a sequence of events in some probability
space, and let $D$ be a dependency graph for $(A_i)_i$. Let $\Delta:=\Delta(D)$ and
suppose that for every $i$ we have $\Pr\left[A_i\right]\leq q$, such that $eq(\Delta+1)<1$. Then, $\Pr[\bigcap_{i=1}^n \bar{A}_i]\geq \left(1-\frac{1}{\Delta+1}\right)^{n}$.
\end{lemma}

We will also make use of the following asymmetric version of the Lov\'asz
Local Lemma (see, e.g., \cite{AlonSpencer}).

\begin{lemma}[Asymmetric Local Lemma]\label{ALLL}
Let $(A_i)_{i=1}^n$ be a sequence of events in some probability
space. Suppose that $D$ is a dependency graph for $(A_i)_i$, and
suppose that there are real numbers $(x_i)_{i=1}^n$, such that $0\leq x_{i} <1$ and 
$$\Pr[A_i]\leq x_i\prod_{ij\in E(D)}(1-x_j)$$
for all $1\leq i\leq n$. 
Then, $\Pr[\bigcap_{i=1}^n \bar{A}_i]\geq\prod_{i=1}^{n}(1-x_{i})$.
\end{lemma}

\subsection{Perfect matchings in bipartite graphs}

Here, we present a number of results related to  perfect matchings in bipartite graphs. The first result is a slight reformulation of the classic Hall's marriage theorem (see, e.g., \cite{SS}).
\begin{theorem}
\label{Hall}
Let $G=(A\cup B,E)$ be a balanced bipartite graph with $|A|=|B|=k$. Suppose $|N(X)|\geq |X|$ for all subsets $X$ of size at most $k/2$ which are completely contained either in $A$ or in $B$. Then, $G$ contains a perfect matching. 
\end{theorem}
Moreover, we can always find a maximum matching in a bipartite graph in polynomial time using standard network flow algorithms (see, e.g., \cite{West}). 

The following simple corollaries of Hall's theorem will be useful for us.
\begin{corollary}\label{matching in regular}
Every $r$-regular balanced bipartite graph has a perfect matching, provided that $r\geq 1$. 
\end{corollary}

\begin{proof}
Let $G=(A\cup B,E)$ be an $r$-regular graph. Let $X\subseteq A$ be a set of size at most $|A|/2$. Note that as $G$ is $r$-regular, we have $$e_G(X,N(X))=r|X|.$$ 
Since each vertex in $N(X)$ has degree at most $r$ into $X$, we get 
$$|N(X)|\geq e_G(X,N(X))/r\geq |X|.$$ 
Similarly, for every $X\subseteq B$ of size at most $|B|/2$ we obtain 
$$|N(X)|\geq |X|.$$
Therefore, by \cref{Hall}, we conclude that $G$ contains a perfect matching.
\end{proof}

Since removing an arbitrary perfect matching from a regular balanced bipartite graph leads to another regular balanced bipartite graph, a simple repeated application of \cref{matching in regular} 
shows the following:

\begin{corollary}
\label{Hall-application-1fact-reg-bipartite}
Every regular balanced bipartite graph has a $1$-factorization. 
\end{corollary}

In fact, as the following theorem due to Schrijver \cite{Schrijver} shows, a regular balanced bipartite graph has many 1-factorizations. 

\begin{theorem}
\label{Schrijver-lower-bound-1-fact}
The number of $1$-factorizations of a $d$-regular bipartite graph with $2k$ vertices is at least 
$$\left(\frac{d!^2}{d^d}\right)^k.$$
\end{theorem}
The next result is a criterion for the existence of $r$-\emph{factors} (that is, $r$-regular, spanning subgraphs) in bipartite graphs, which follows from a generalization of the Gale-Ryser theorem due to Mirsky \cite{Mirsky}. 

\begin{theorem}
  \label{gale}
Let $G=(A\cup B,E)$ be a balanced bipartite graph with $|A|=|B|=m$, and let $r$ be an integer. Then, $G$ contains an $r$-factor if and only if for all $X\subseteq A$ and $Y\subseteq B$
\begin{align*}
  e_G(X,Y)\geq r(|X|+|Y|-m).
\end{align*}
\end{theorem}
Moreover, such factors can be found efficiently using standard network flow algorithms (see, e.g., \cite{Anstee}).  

As we are going to work with pseudorandom graphs, it will be convenient for us to isolate some `nice' properties that, together with \cref{gale}, ensure the existence of large factors in balanced bipartite graphs.

\begin{lemma}
\label{finding-large-factors}
Let $G=(A\cup B,E)$ be a balanced bipartite graph
with $|A|=|B|=n/2$. Suppose there exist $r,\varphi\in\R^{+}$ and
$\beta_{1},\beta_{2},\beta_{3},\gamma\in(0,1)$ satisfying the following
additional properties:
\begin{enumerate} [$(P1)$]
\item $\deg_{G}(v)\geq r(1-\beta_{1})$ for all $v\in A\cup B$.
\item $e_{G}(X,Y)< r\beta_{2}|X|$ for all $X\subseteq A$ and $Y\subseteq B$
with $|X|=|Y|\leq r/\varphi$.
\item $e_{G}(X,Y)\geq2r(1-\beta_{3})|X||Y|/n$ for all $X\subseteq A$
and $Y\subseteq B$ with $|X|+|Y|>n/2$ and $\min\{|X|,|Y|\}>r/\varphi$.
\item $\gamma\geq\max\{\beta_{3},\beta_{1}+\beta_{2}\}$
\end{enumerate}
Then, $G$ contains an $\lfloor r(1-\gamma)\rfloor$-factor. 
\end{lemma}
\begin{proof} By \cref{gale}, it suffices to verify
that for all $X\subseteq A$ and $Y\subseteq B$ we have
$$e_{G}(X,Y)\geq r(1-\gamma)\left(|X|+|Y|-\frac{n}{2}\right).$$
We divide the analysis into five cases:
\begin{enumerate} [$\text{Case }1$]

\item $|X|+|Y|\leq n/2$. In this case, we trivially have
$$e_{G}(X,Y)\geq 0\geq r(1-\gamma)\left(|X|+|Y|-\frac{n}{2}\right),$$ so there is nothing to prove.

\item $|X|+|Y|>n/2$ and $|X|\leq r/\varphi$. Since
$|Y|\leq|B|=\frac{n}{2}$, we always have $|X|+|Y|-\frac{n}{2}\geq|X|$.
Thus, it suffices to verify that
$$e_{G}(X,Y)\geq r(1-\gamma)|X|.$$
Assume, for the sake of contradiction, that this is not the case. Then, since there
are at least $r(1-\beta_{1})|X|$ edges incident to $X$, we must
have
$$e_{G}(X,B\backslash Y)\geq r(1-\beta_1)|X|-e_G(X,Y)\geq r(\gamma-\beta_{1})|X|\geq r\beta_{2}|X|.$$
However, since $|B\backslash Y|\leq|X|$, this contradicts $(P2)$.

\item $|X|+|Y|>n/2$ and $|Y|\leq r/\varphi$. This is
exactly the same as the previous case with the roles of $X$ and $Y$ interchanged.

\item $|X|+|Y|>n/2$, $|X|,|Y|>r/\varphi$ and $|Y|\geq|X|$.
By $(P3)$, it suffices to verify that
$$2r(1-\beta_{3})|X||Y|/n\geq r(1-\gamma)\left(|X|+|Y|-n/2\right).$$
Dividing both sides by $rn/2$, the above inequality is equivalent to $$xy-\frac{(1-\gamma)}{(1-\beta_{3})}(x+y-1)\geq0,$$
where $x=2|X|/n$, $y=2|Y|/n$,  $x+y\geq1$, $0\leq x\leq1$, and $0\leq y\leq1$.

Since $\frac{1-\gamma}{1-\beta_{3}}\leq1$ by $(P4)$, this is readily verified
on the (triangular) boundary of the region, on which the inequality reduces to one of the following: $xy\geq 0$; $x \geq \frac{1-\gamma}{1-\beta_{3}}x$; $y\geq \frac{1-\gamma}{1-\beta_{3}}y$. On the other hand, the
only critical point in the interior of the region is possibly at $x_{0}=y_{0}=\frac{1-\gamma}{1-\beta_{3}}$,
for which we have $x_{0}y_{0}-\frac{1-\gamma}{1-\beta_{3}}(x_{0}+y_{0}-1)=\frac{1-\gamma}{1-\beta_{3}}\left(1-\frac{1-\gamma}{1-\beta_{3}}\right)\geq0$,
again by $(P4)$.

\item $|X|+|Y|>n/2$, $|X|,|Y|>r/\varphi$ and $|X|\geq|Y|$.
This is exactly the same as the previous case with the roles of $X$ and $Y$
interchanged.
\end{enumerate}
\end{proof}

\subsection{Matchings in graphs with controlled degrees}
In this section, we collect a couple of results on matchings in (not necessarily bipartite) graphs satisfying some degree conditions. A $2$-\emph{factorization} of a graph is a decomposition of its edges into $2$-factors (that is, a collection of vertex disjoint cycles that covers all the vertices). The following theorem, due to Petersen \cite{Pet}, is one of the earliest results in graph theory.   
\begin{theorem}[2-factor Theorem] 
\label{2-factor theorem}
Every $2k$-regular graph with $k\geq 1$ admits a 2-factorization. 
\end{theorem}

The next theorem, due to Vizing \cite{Vizing}, shows that every graph $G$ admits a proper edge coloring using at most $\Delta(G)+1$ colors.  
\begin{theorem}[Vizing's Theorem] 
\label{Vizing's theorem}
Every graph with maximum degree $\Delta$ 
can be properly edge-colored with $k\in\{\Delta,\Delta+1\}$ colors. 
\end{theorem}


\subsection{Expander Mixing Lemma}
When dealing with $(n,d,\lambda)$ graphs, we will repeatedly use the following lemma (see, e.g., \cite{HLW}), which bounds the difference between the actual number of edges between two sets of vertices, and the number of edges we expect based on the sizes of the sets. 

\begin{lemma}[Expander Mixing Lemma]
\label{Expander Mixing Lemma}
Let $G=(V,E)$ be an $(n,d,\lambda)$ graph, and let $S,T\subseteq V$. Let $e(S,T) = |\{(x,y)\in S\times T: xy \in E\}|$. Then, 
$$\left|e(S,T) - \frac{d|S||T|}{n}\right|\leq \lambda \sqrt{|S||T|}.$$
\end{lemma}

\section{Random partitioning}
\label{section-random-partitioning}
While we have quite a few easy-to-use tools for working with balanced bipartite graphs, the graph we start with is not necessarily bipartite (when the starting graph \emph{is} bipartite, the existence problem is easy, and the counting problem is solved by \cite{Schrijver}). Therefore, perhaps the most natural thing to do is to partition the edges into `many' balanced bipartite graphs, where each piece has suitable expansion and regularity properties. The following lemma is our first step towards obtaining such a partition.  

\begin{lemma}
\label{lemma:finding initial partitions}
Fix $a \in (0,1)$, and let $G=(V,E)$ be a $d$-regular graph on $n$ vertices, where $d=\omega_{a}(1)$ 
and $n$ is an even and sufficiently large integer. Then, for every integer $t\in [d^{a/100}, d^{1/10}]$, there exists a collection $(A_i,B_i)_{i=1}^t$ of balanced bipartitions for which the following properties hold:
\begin{enumerate}[$(R1)$]
\item Let $G_i$ be the subgraph of $G$ induced by $E_G(A_i,B_i)$. For all $1\leq i\leq t $ we have
$$\frac{d}{2}-d^{2/3}\leq \delta(G_i)\leq \Delta(G_i)\leq \frac{d}{2}+d^{2/3}.$$
\item For all $e\in E(G)$, the number of indices $i$ for which $e\in E(G_i)$ is $\frac{t}{2}\pm t^{2/3}$.
\end{enumerate}
\end{lemma}
We will divide the proof into two cases -- the \emph{dense case}, where $\log^{1000}{n}\leq d \leq n-1$, and the \emph{sparse case}, where $ d \leq \log^{1000}{n} $. The underlying idea is similar in both cases, but the proof in the sparse case is technically more involved as a standard use of Chernoff's bounds and the union bound does not work (and therefore, we will instead use the Local Lemma).

\begin{proof}[Proof in the dense case]
Let $A_{1},\dots,A_{t}$ be random subsets chosen independently from the uniform distribution on all subsets of $[n]$ of size exactly $n/2$, and let $B_{i}=V\setminus A_{i}$ for all $1\leq i\leq t$. We will show that with high probability, for every $1\leq i\leq t$, $(A_{i},B_{i})$ is a balanced bipartition satisfying $(R1)$ and $(R2)$.

First, note that for any $e\in E(G)$ and any $i \in [t]$, $$\Pr\left[e\in E(G_{i})\right]=\frac{1}{2}\left(1+\frac{1}{n-1}\right).$$ Therefore, if for all $e\in E(G)$ we let $C(e)$ denote the set of indices $i$ for which $e\in E(G_{i})$, then 
$$\E \left[|C(e)|\right] = \frac{t}{2}\left(1+\frac{1}{n-1}\right).$$ 

Next, note that, for a fixed $e\in E(G)$, the events $A_i:=`i\in C(e)$' are mutually independent, and that $|C(e)|=\sum_i X_i$, where $X_i$ is the indicator random variable for the event $A_i$. Therefore, by Chernoff's bounds, it follows that   
$$\Pr\left[|C(e)|\notin \frac{t}{2} \pm t^{2/3}\right]\leq \exp\left(-t^{1/3}\right)\leq \frac{1}{n^3}.$$ 

Now, by applying the union bound over all $e \in E(G)$, it follows that the collection $(A_{i},B_{i})_{i=1}^{t}$ satisfies $(R2)$ with probability at least $1-1/n$. Similarly, it is immediate from Chernoff's inequality for the hypergeometric distribution that for any $v\in V$ and $i \in [t]$, $$\Pr\left[d_{G_i}(v)\notin \frac{d}{2} \pm d^{2/3}\right] \leq \exp\left(-\frac{d^{1/3}}{10}\right)\leq \frac{1}{n^3},$$ and by taking the union bound over all such $i$ and $v$, it follows that $(R1)$ holds with probability at least $1-1/n$. All in all, with probability at least $1-2/n$, both $(R1)$ and $(R2)$ hold. This completes the proof.
\end{proof}

\begin{proof}[Proof in the sparse case]
Instead of using the union bound as in the dense case, we will use the symmetric version of the Local Lemma (\cref{LLL}). 
Note that there is a small obstacle with choosing \emph{balanced} bipartitions, as the local lemma is most convenient to work with when the underlying experiment is based on independent trials. In order to overcome this issue, we start by defining an auxiliary graph $G'=(V,E')$ as follows: for all $xy\in \binom{V}{2}$, $xy\in E'$ if and only if $xy \notin E$ and there is no vertex $v\in V(G)$ with $\{x,y\}\subseteq N_G(v)$. In other words, there is an edge between $x$ and $y$ in $G'$ if and only if $x$ and $y$ are not connected to each other, and do not have any common neighbors in $G$.  
Since for any $x\in V$, there are at most $d^{2}$ many $y \in V$ such that $xy \in E$ or $x$ and $y$ have a common neighbor in $G$, it follows that $\delta(G') \geq n-d^{2}\geq n/2$ for $n$ sufficiently large. An immediate application of Hall's theorem shows that any graph on $2k$ vertices with minimum degree at least $k$ contains a perfect matching. Therefore, $G'$ contains a perfect matching. 

Let $s=n/2$ and let $M:=\{x_1y_1,\ldots,x_{s}y_s\}$ be an arbitrary perfect matching of $G'$.
For each $i\in[t]$ let $f_{i}$ be a random function chosen independently and uniformly from the set of all functions from $\{x_{1},\dots,x_{s}\}$ to $\{\pm 1\}$. These functions will define the partitions according to the vertex labels as follows: 
$$A_i:=\{x_j \mid f_i(x_j)=-1\}\cup \{y_j \mid f_i(x_j)=+1\},$$ and
$$B_i:=[n]\setminus V_i.$$ 
Clearly, $(A_i, B_i)_{i=1}^{t}$ is a random collection of balanced bipartitions of $V$. 
If, for all $i\in[t]$, we let $g_{i}\colon V(G)\to \{A,B\}$ denote the random function recording which of $A_i$ or $B_i$ a given vertex ends up in, it is clear -- and this is the point of using $G'$ -- that for any $i \in [t]$ and any $v\in V(G)$, the choices $\{g_i(w)\}_{w\in N_G(v)}$ are mutually independent. This will help us in showing that, with positive probability, this collection of bipartitions satisfies properties $(R1)$ and $(R2)$. 

Indeed, for all $v\in V(G)$, $i\in [t]$, and $e \in E(G)$, let 
$D_{i,v}$ denote the event that `$d_{G_i}(v)\notin \frac{d}{2}\pm d^{2/3}$', and let $A_{e}$ denote the event `$|C(e)| \notin \frac{t}{2} \pm t^{2/3}$'. Then, using the independence property mentioned above, Chernoff's bounds imply that 
$$\Pr[D_{i,v}] \leq \exp\left(-d^{1/3}/4\right)$$ and 
$$\Pr[A_{e}] \leq\exp\left(-t^{1/3}/4\right).$$ 

In order to complete the proof, we need to show that one can apply the symmetric local lemma (\cref{LLL}) to the collection of events consisting of all the $D_{i,v}$'s and all the $A_{e}$'s. To this end, we first need to upper bound the number of events which depend on any given event.

Note that $D_{i,v}$ depends on $D_{j,u}$ only if $i=j$ and $\text{dist}_G(u,v)\leq 2$ or $uv\in M$. Note also that $D_{i,v}$ depends on $A_{e}$ only if an end point of $e$ is within distance $1$ of $v$ either in $G$ or in $M$. Therefore, it follows that any $D_{i,v}$ can depend on at most $2d^{2}$ events in the collection. Since $A_{e}$ can depend on $A_{e'}$ only if $e$ and $e'$ share an endpoint in $G$ or if any of the endpoints of $e$ are matched to any of the endpoints in $M$, it follows that we can take the maximum degree of the dependency graph in \cref{LLL} to be $2d^{2}$. Since $2d^{2}\exp(-t^{1/3}/4)=o_d(1/e)$, we are done.      
\end{proof}

\section{Completion}
In this section, we describe the key ingredient of our proof, namely the completion step. Before stating the relevant lemma, we need the following definition. 
\begin{definition}
\label{defn-good-graph}
A graph $H=(A\cup B,E)$ is called $(\alpha,r,m)$-good if it
satisfies the following properties: 
\begin{enumerate}[$(G1)$]
\item $H$ is an $r$-regular, balanced bipartite graph with $|A|=|B|=m$.
\item Every balanced bipartite subgraph $H'=(A'\cup B',E')$ of $H$ with
$|A'|=|B'|\geq(1-\alpha)m$ and with $\delta(H')\geq(1-2\alpha)r$
contains a perfect matching. 
\end{enumerate}
\end{definition}
The motivation for this definition comes from the next proposition, which shows that a regular graph on an even number of vertices, which can be decomposed as a union of a good graph and a sufficiently sparse graph, has a 1-factorization. 
\begin{proposition}
\label{proposition-completion}
For every $\alpha\leq 1/10$, there exists an integer $r_0$ such that for all $r_0 \leq r_1$ and $m$ a sufficiently large integer, the following holds. Suppose that $H=(A\cup B, E(H))$ is an $(\alpha,r_1,m)$-good graph. Then, for every $r_2\leq \alpha^5r_1/\log{r_1}$, every $r:=r_1 + r_2$-regular (not necessarily bipartite) graph $R$ on the vertex set $A\cup B$, for which $H\subseteq R$, admits a 1-factorization. 
\end{proposition}
For clarity of exposition, we first provide the somewhat simpler proof (which already contains all the ideas) of this proposition in the `dense' case, and then we proceed to the `sparse' case.
\begin{proof}[Proof in the dense case: $m^{1/10}\leq r_1\leq m$]
First, observe that $e(R[A])=e(R[B])$. Indeed, as $R$ is $r$-regular, we have for $X\in \{A,B\}$ that 
$$rm=\sum_{v\in X}d_R(v)=2e(R[X])+e(R[A,B]),$$ from which the above equality follows. Moreover, $\Delta(R[X])\leq r_2$ for all $X\in \{A,B\}$.
Next, let $R_0:=R$ and $f_0:=e(R_0[A])=e(R_0[B])$.  
By Vizing's theorem (\cref{Vizing's theorem}), both $R_0[A]$ and $R_0[B]$ contain matchings of size exactly $\lceil f_0/(r_2+1)\rceil$. Consider any two such matchings $M_A$ in $A$ and $M_B$ in $B$, and for $X\in \{A,B\}$, let $M'_X\subseteq M_X$, 
denote a matching of size $|M'_X| =\lfloor \alpha f_0/2r_2 \rfloor$ such that no vertex $v \in V(H)$ is incident to more than $3\alpha r_1/2$ vertices which are paired in the matching. To see that such an $M'_X$ must exist, we use a simple  probabilistic argument -- for a random subset $M'_X\subseteq M_X$ of this size, by a simple application of Hoeffding's inequality and the union bound, we obtain that $M'_X$ satisfies the desired property, except with probability at most $2m\exp(-\alpha^{2}r_1/8)\ll 1$.    

Delete the \emph{vertices} in $\left(\cup M'_A\right)\bigcup \left(\cup M'_B\right)$, as well as any edges incident to them, from  $H$ and denote the resulting graph by $H'=(A'\cup B',E')$. Since $|A'|=|B'|\geq (1-\alpha)|A|$ and $\delta(H')\geq (1-3\alpha/2)r_1$ by the choice of $M'_X$, it follows from $(G2)$ that $H'$ contains a perfect matching $M'$. Note that $M_0:=M'\cup M'_A\cup M'_B$ is a perfect matching in $R_0$. We repeat this process with $R_1:=R_0 - M_1$ (deleting only the edges in $M_1$, and not the vertices) and $f_1 :=e(R_1[A])=e(R_1[B])$ until we reach $R_k$ and $f_k$ such that $f_k\leq r_2$. Since $f_{i+1} \leq \left(1-\alpha/3r_2\right)f_i$, this must happen after at most $3r_2\log{m}/\alpha \ll \alpha^{2} r_1$ steps.  Moreover, since $\deg(R_{i+1}) = \deg(R_i)-1$, it follows that during the first $\lceil 3r_2\log {m}/\alpha\rceil$ steps of this process, the degree of any $R_j$ is at least $r_1 - \alpha^2 r_1$. Therefore, since $(r_1 - \alpha^2r_1) - 3\alpha r_1/2 \geq (1-2\alpha)r_1$, we can indeed use $(G2)$ throughout the process, as done above.  

From this point onwards, we continue the above process (starting with $R_k$) with matchings of size one i.e. single edges from each part, until no more edges are left. By the choice of $f_k$, we need at most $r_2$ such iterations, which is certainly possible since $r_2 + 3r_2\log{m}/\alpha \ll \alpha^{2} r_1$. After removing all the perfect matchings obtained via this procedure, we are left with a regular, balanced, \emph{bipartite} graph, and therefore it admits a 1-factorization (\cref{Hall-application-1fact-reg-bipartite}). Taking any such 1-factorization along with all the perfect matchings that we removed gives a 1-factorization of $R$.     
\end{proof}
\begin{proof}[Proof in the sparse case: $r_1\leq m^{1/10}$]
Let $C$ be any integer and let $k = \lfloor 1/\alpha^{4}\rfloor$. 
We begin by showing that any matching $M$ in $X\in\{A,B\}$ with $|M|=C$ can be partitioned into $k$ matchings $M_1,\dots, M_k$ 
such that no vertex $v\in V(H)$ is incident to more than $\alpha r_1$ vertices in $\cup M_i$ for any $i\in[k]$. If $C<\alpha r_1/2$, then there is nothing to show. If $C \geq \alpha r_1/2$, consider an arbitrary partition of $M$ into $\lceil C/k\rceil $ sets $S_1,\dots,S_{\lceil C/k\rceil}$ 
with each set (except possibly the last one) of size $k$. For each $S_{j}$, $j\in \lfloor C/k \rfloor$, choose a permutation of $[k]$ independently and uniformly at random, and let $M_i$ denote the random subset of $M$ consisting of all elements of $S_1,\dots,S_{\lceil C/k \rceil}$ which are assigned the label $i$. We will show, using the symmetric version of the Local Lemma (\cref{LLL}), that the decomposition $M_1,\dots,M_k$ satisfies the desired property with a positive probability.

To this end, note that for any vertex $v$ to have at least $\alpha r_1$ neighbors in some $M_i$, it must be the case that the $r_1$ neighbors of $v$ in $H$ are spread throughout at least $\alpha r_1$ distinct $S_j$'s. Let $D_v$ denote the event that $v$ has at least $\alpha r_1$ neighbors in some matching $M_i$. Since $v$ has at least $\sqrt{k}$ neighbors in at most $r_1/\sqrt{k} \ll \alpha r_1$ distinct $S_j$'s, it follows that $\Pr[D_v] \leq k(1/\sqrt{k})^{\alpha r_1/2}$. Finally, since each $D_v$ depends on at most $r^{2}k$ many other $D_w$'s, and since $k^{2}r^{2}(1/\sqrt{k})^{\alpha r_1/2} <1/e$, we are done.

Now, as in the proof of the dense case, we have $e(R[A])=e(R[B])$ and $\Delta(R[X])\leq r_2$ for $X\in\{A,B\}$. By Vizing's theorem, we can decompose $R[A]$ and $R[B]$ into exactly $r_2+1$ matchings each, and it is readily seen that these matchings can be used to decompose $R[A]$ and $R[B]$ into at most $\ell \leq 2(r_2+1)$ matchings $M^{A}_1,\dots,M^{A}_\ell$ and $M^{B}_1,\dots,M^{B}_\ell$ such that $|M^{A}_i|=|M^{B}_i|$ for all $i\in[\ell]$. Using the argument in the previous paragraph, we can further decompose each $|M^{X}_i|$, $X\in\{A,B\}$ into at most $k$ matchings each such that no vertex $v\in V(H)$ is incident to more than $\alpha r_1$ vertices involved in any of these smaller matchings. Since $4r_2/\alpha^{4} \ll r_1$, the rest of the argument proceeds exactly as in the dense case.     
\end{proof}
\begin{remark}
\label{rmk:counting-completion}
In the last step of the proof, we are allowed to choose an arbitrary $1$-factorization of an $r'$-regular, balanced bipartite graph, where $r'\geq r_1 - r_2$. Therefore, using \cref{Schrijver-lower-bound-1-fact} along with the standard inequality $d! \geq (d/e)^d$, it follows that $R$ admits at least $\left((r_{1}-r_{2})/e^{2}\right)^{(r_{1}-r_{2})m}$ $1$-factorizations. 
\end{remark}

\section{Finding good subgraphs which almost cover $G$}
\label{section-structuralresult}
In this section we present a structural result which shows that a  `good' regular expander on an even number of vertices can be `almost' covered by a union of edge disjoint good subgraphs.  
\begin{proposition}
\label{existence-good-graph}
For every $c>0$ there exists $d_0$ such that for all $d\geq d_0$ the following holds. Let $G=(V,E)$ be an $(n,d,\lambda)$-graph with $\lambda < d/4t^{4}$
where $t$ is an integer in $[d^{c/100},d^{c/10}] $. 
Then, $G$ contains $t$ distinct, edge disjoint $\left(\alpha, \lfloor\bar{r}\rfloor, \frac{n}{2}\right)$-good subgraphs $W_{1},\dots,W_{t}$ with $\alpha=\frac{1}{10}$ and $\bar{r}=\frac{d}{t}\left(1-\frac{16}{t^{1/3}}\right)$.
\end{proposition}

The proof of this proposition is based on the following technical lemma, which lets us apply \cref{gale} to each part of the partitioning coming from \cref{lemma:finding initial partitions} in order to find large good factors. 
\begin{lemma}
\label{lemma-almost-factorization-technical}
There exists an edge partitioning $E(G)=E_1\cup \ldots E_{t}$ for which the following properties hold:
\begin{enumerate}[$(S1)$]
\item $H_i:=G[E_i]$ is a balanced bipartite graph with parts $(A_i,B_i)$ for all $i\in[t]$. 
\item For all $i\in[t]$ and for all $X\subseteq A_{i}$, $Y\subseteq B_{i}$ with $|X|=|Y| \leq n/2t^{2}$ we have $$e_{H_i}(X,Y) < d|X|/t^{2}.$$ 
\item For all $i\in [t]$ and all $X\subseteq A_i$, $Y\subseteq B_i$ with $|X|+|Y|>n/2$ and $\min\{|X|,|Y|\}>\frac{n}{2t^{2}}$, $e_{H_i}(X,Y) \geq 2\frac{d}{t}\left(1-\frac{8}{t^{1/3}}\right)\frac{|X||Y|}{n}.$ 
\item $d_{H_i}(v)\in \frac{d}{t}\pm \frac{8d}{t^{4/3}}$ for all $i\in[t]$ and all $v\in V(H_i)=V(G)$. 
\item $e_{H_i}(X,Y)\leq(1-4\alpha)\frac{d}{t}|X|$ for all $X,Y\subseteq V(H_i)$ with $\frac{n}{2t^{2}} \leq |X|=|Y|\leq \frac{n}{4}$.
\end{enumerate}
\end{lemma}
Before proving this lemma, let us show how it can be used to prove \cref{existence-good-graph}.

\begin{proof}[Proof of \cref{existence-good-graph}]
Note that each balanced bipartite graph $H_1,\dots,H_t$ coming from \cref{lemma-almost-factorization-technical} satisfies the hypotheses of \cref{gale} with 
$$r=\frac{d}{t}, \varphi = \frac{2rt^{2}}{n}, \beta_{1}=\beta_{2}=\beta_{3}=\frac{8}{t^{1/3}}, \gamma=\frac{16}{t^{1/3}}.$$
Indeed, $(P1)$ follows from $(S4)$, $(P2)$ follows from $(S2)$, $(P3)$ follows from $(S3)$ and $(P4)$ is satisfied by the choice of parameters. Therefore, \cref{gale} guarantees that each $H_i$ contains an $\lfloor \bar{r}\rfloor$-factor, and by construction, these are edge disjoint.

Now, let $W_1,\dots,W_t$ be any $\lfloor\bar{r} \rfloor$-factors of $H_1,\dots,H_t$. 
It remains to check that $W_1,\dots,W_t$ satisfy property $(G2)$. We will actually show the stronger statement that $H_1,\dots,H_t$ satisfy $(G2)$. 
Indeed, let $H'_i=(A'_i\cup B'_i,E'_i)$ be a subgraph of $H_i$ with $A'_i\subseteq A_i$, $B'_i\subseteq B_i$ such that 
$$|A'_i|=|B'_i|\geq(1-\alpha)n/2$$ and 
$$\delta(H'_i)\geq(1-2\alpha)\lfloor\bar{r}\rfloor.$$
Suppose $H'_i$ does not contain a perfect matching. Then, by \cref{Hall}, without loss of generality, there must exist $X\subseteq A_i$ and
$Y\subseteq B_i$ such that 
$$|X|=|Y|\leq|A_i|/2\leq n/4$$ and 
$$N_{H'_i}(X)\subseteq Y.$$ 
In particular, by the minimum degree assumption, it follows that 
$$e_{H'_i}(X,Y)\geq(1-2\alpha)\lfloor\bar{r}\rfloor|X|.$$
On the other hand, we know that for such a pair, 
$$e_{H'_i}(X,Y)\leq e_{G}(X,Y)\leq d|X|^{2}/n+\lambda|X|.$$ 
Thus, we must necessarily have that $|X| \geq n/2t^{2}$, which contradicts $(S5)$. This completes the proof. 
\end{proof}

\begin{proof}[Proof of \cref{lemma-almost-factorization-technical}]
Our construction will be probabilistic. We begin by applying \cref{lemma:finding initial partitions} to $G$ to obtain a collection of balanced bipartitions $(A_i,B_i)_{i=1}^{t}$ satisfying Properties $(R1)$ and $(R2)$ of \cref{lemma:finding initial partitions}. Let $G_i:=G[A_i,B_i]$, and for each $e\in E(G)$, let $C(e)$ denote the set of indices $i\in[t]$ for which $e\in E(G_i)$. Let $\{c(e)\}_{e\in E(G)}$ denote a random collection of elements of $[t]$, where each $c(e)$ is chosen independently and uniformly at random from $C(e)$.
Let $H_i$ be the (random) subgraph of $G_i$ obtained by keeping all the edges $e$ with $c(e)=i$. Then, the $H_i$'s always form an edge partitioning of $E(G)$ into $t$ balanced bipartite graphs with parts $(A_i,B_i)_{i=1}^{t}$. 

It is easy to see that these $H_i$'s will always satisfy $(S2)$. Indeed, if for any $X,Y\subseteq V(G)$ with $|X|=|Y|$, we have $e_{H_i}(X,Y)\geq d|X|/t^{2}$,
then since 
$$e_{H_i}(X,Y)\leq e_G(X,Y)\leq \frac{d|X|^{2}}{n}+\lambda |X|$$
by the Expander Mixing Lemma (\cref{Expander Mixing Lemma}), it follows that 
$$\frac{d}{t^{2}}\leq \frac{d|X|}{n} +\lambda,$$ and therefore, we must have 
$$|X|> \frac{n}{2t^{2}}.$$

We now provide a lower bound on the probability with which this partitioning also satisfies $(S3)$ and $(S4)$. 
To this end, we first define the following events: 
\begin{itemize}
\item For all $v\in V(G)$ and $i\in [t]$, let $D_{i,v}$ denote the event that $d_{H_i}(v)\notin \frac{d}{t}\pm \frac{8d}{t^{4/3}}$. 
\item For all $i\in [t]$ and all $X\subseteq A_i$, $Y\subseteq B_i$ with $|X|+|Y|>n/2$ and $\min\{|X|,|Y|\}>\frac{n}{2t^{2}}$, let $A(i,X,Y)$ denote the event that $e_{H_i}(X,Y)\leq 2\frac{d}{t}\left(1-\frac{8}{t^{1/3}}\right)\frac{|X||Y|}{n}$
\end{itemize}


Next, we wish to upper bound the probability of occurrence for each of these events. 

Note that for all $i\in[t]$ and $v\in V(G)$, it follows from $(R1)$ and $(R2)$ that $$\mathbb{E}[d_{H_i}(v)]\in \frac{d/2\pm d^{2/3}}{t/2\mp t^{2/3}}\in \frac{d}{t} \pm \frac{4d}{t^{4/3}}.$$

Therefore, by Chernoff's inequality, we get that for all $i\in[t]$ and $v\in V(G)$, 
\begin{equation}
\label{bound on Dv}
\Pr[D_{i,v}]\leq \exp\left(-\frac{d}{t^{5/3}}\right).
\end{equation}

Moreover, for all $i\in[t]$ and for all $X\subseteq A_i$, $Y\subseteq B_i$ with $|X|+|Y|>n/2$ and $\min\{|X|,|Y|\}>\frac{n}{2t^{2}}$, we have from the expander mixing lemma and $(R2)$ that 
$$\mathbb{E}[e_{H_i}(X,Y)]\geq2\frac{d}{t}\left(1-\frac{4}{t^{1/3}}\right)\frac{|X||Y|}{n}.$$



Therefore, by Chernoff's bounds, we get that for $i\in[t]$ and all such $X,Y$,
\begin{equation}
\label{bound on edges between X and Y}
\Pr\left[A(i,X,Y)\right]\leq \exp\left(-\frac{d|X||Y|}{nt^{8/3}}\right).
\end{equation}

Now, we apply the asymmetric version of the Local Lemma (\cref{ALLL}) as follows: our events consist of all 
the previously defined $D_{i,v}$'s and $A(j,X,Y)$'s. 
Note that each $D_{i,v}$ depends only on those $D_{j,w}$ for which $\text{dist}_{G}(v,w)\leq 2$. In particular, each $D_{i,v}$ depends on at most $td^{2}$ many $D_{j,w}$. Moreover, we assume that $D_{i,v}$ depends on all the events $A(j,X,Y)$ and that each $A(j,X,Y)$ depends on all the other events. 
For convenience, let us enumerate all the events as $\mathcal E_k$, $k=1,\ldots \ell$. For each $k\in[\ell]$, let $x_k$ be $\exp\left(-\sqrt{d}\right)$ if $\mathcal E_k$ is of the form $D_{j,v}$, and $x_k$ be $\exp\left(-\sqrt{d}|X||Y|/n\right)$ if $\mathcal E_k$ is of the form $A(j,X,Y)$.
To conclude the proof, we verify that $$\Pr[\mathcal E_k]\leq x_k\prod_{j\sim k}(1-x_j)$$ for all $k$. 
Indeed, if $\mathcal E_k$ is of the form $D_{j,v}$ then we have
\begin{align*}
e^{-\sqrt{d}}\left(1-e^{-\sqrt{d}}\right)^{td^{2}}\prod_{x,y}\left(1-e^{-\sqrt{d}xy/n}\right)^{{n \choose x}{n \choose y}} & \geq e^{-\sqrt{d}}e^{-2td^{2}e^{-\sqrt{d}}}\prod_{x,y}\left(e^{-2e^{-\sqrt{d}xy/n}}\right)^{{n \choose x}{n \choose y}}\\
 & \geq e^{-\sqrt{d}}e^{-2td^{2}/e^{\sqrt{d}}}\prod_{x,y}\exp\left(-2e^{-\sqrt{d}n/8t^{2}}{n \choose x}{n \choose y}\right)\\
 & \geq e^{-\sqrt{d}}e^{-2td^{2}/e^{\sqrt{d}}}\exp\left(-e^{-\sqrt{d}n/8t^{2}}2^{3n}\right)\\
 & >\Pr[\mathcal{E}_{k}].
\end{align*}


On the other hand, if $\mathcal E_k$ is of the form $A(j,X,Y)$, then we have 
\begin{align*}
e^{-\sqrt{d}xy/n}(1-e^{-\sqrt{d}})^{nt}\prod_{x,y}(1-e^{-\sqrt{d}xy/n})^{{n \choose x}{n \choose y}} & \geq e^{-\sqrt{d}xy/n}e^{-2e^{-\sqrt{d}}nt}\prod_{x,y}\left(e^{-2e^{-\sqrt{d}xy/n}}\right)^{{n \choose x}{n \choose y}}\\
 & \geq e^{-\sqrt{d}xy/n-2e^{-\sqrt{d}}nt}\prod_{x,y}\exp\left(-2e^{-\sqrt{d}n/8t^{2}}{n \choose x}{n \choose y}\right)\\
 & \geq e^{-\sqrt{d}xy/n-2e^{-\sqrt{d}}nt}\exp\left(-e^{-\sqrt{d}n/8t^{2}}2^{3n}\right)\\
 & >\Pr[\mathcal{E}_{k}].
\end{align*}
Therefore, by the asymmetric version of the Local Lemma, Properties $(S3)$ and $(S4)$ are satisfied with probability at least $$\left(1-e^{-\sqrt{d}}\right)^{nt}\prod_{x,y}(1-e^{-\sqrt{d}xy/n})^{{n \choose x}{n \choose y}}\geq e^{-nt}.$$

To complete the proof, it suffices to show that the probability that $(S5)$ is not satisfied is less than $\exp(-nt)$. To see this, fix $i\in [t]$ and $X,Y\subseteq V(H_i)$ with  $n/2t^{2} \leq|X|=|Y|\leq n/4$. By the expander mixing lemma,  
we know that $$e_{G}(X,Y)\leq d|X|^{2}/n+ \lambda|X| \leq d|X|/4 + \lambda |X|,$$ so by $(R2)$ we get  
$$\mathbb{E}[e_{H_i}(X,Y)]\leq\frac{d}{2t}\left(1+\frac{4}{t^{1/3}}\right)|X|.$$
Therefore,
by Chernoff's bounds, it follows that
$$\Pr\left[e_{H_i}(X,Y)\geq(1-4\alpha)d|X|/t\right] \leq \exp\left(-d|X|/50t\right) \leq \exp\left(-dn/100t^{3}\right).$$
Applying the union bound over all $i\in [t]$, and all such $X,Y\subseteq V(G)$, implies that the probability for $(S5)$ to fail is at most $\exp(-dn/200t^{3})<\exp(-nt)/2$. This completes the proof.
\end{proof}
\begin{remark}
\label{rmk:counting}
The above proof shows that there are at least $\frac{1}{2}\exp(-nt)\left(\frac{t}{2}-t^{2/3}\right)^{nd/2}$ (labeled) edge partitionings satisfying the conclusions of \cref{lemma-almost-factorization-technical}.
\end{remark}

\section{Proofs of \cref{main:pseudorandom} and \cref{main:counting}}
In this section, by putting everything together, we obtain the proofs of our main results. 

\begin{proof}[Proof of \cref{main:pseudorandom}]
Let $c=\varepsilon/10$, and apply \cref{existence-good-graph} with $\alpha=1/10$, $c$, and $t$ being an odd integer in $[d^{c/100},d^{c/10}]$ to obtain $t$ distinct, edge disjoint $\left(\alpha,r,\frac{n}{2}\right)$-good graphs $W_1,\dots,W_t$, where $r=\lfloor\frac{d}{t}\left(1-\frac{16}{t^{1/3}}\right) \rfloor$. Let $G':=G\setminus \bigcup_{i=1}^{t}W_i$, and note that $G'$ is $r':=d-rt$ regular. After possibly replacing $r$ by $r-1$, we may further assume that $r'$ is even. In particular, by \cref{2-factor theorem}, $G'$ admits a 2-factorization. By grouping these 2-factors, we readily obtain a decomposition of $G'$ as $G'=G'_1\cup\dots\cup G'_t$ where each $G'_i$ is $r'_i$-regular, with $r'_i \in \frac{r'}{t} \pm t\leq 40\frac{d}{t^{4/3}}$. Finally, applying \cref{proposition-completion} to each of the regular graphs $R_1,\dots,R_t$, where $R_i:=W_i\cup G'_i$, finishes the proof.  
\end{proof}

We will obtain `enough' 1-factorizations by keeping track of the number of choices available to us at every step in the above proof.

\begin{proof}[Proof of \cref{main:counting}]
Suppose that $\lambda\leq d^{1-\varepsilon}$ and let $c=\varepsilon/10$. Now, fix $\epsilon > 0$. Throughout this proof, $\epsilon_1,\dots, \epsilon_4$ will denote positive quantities which go to $0$ as $d$ goes to infinity. By \cref{rmk:counting}, there are at least $\left((1-\epsilon_1)\frac{t}{2}\right)^{nd/2}$ edge partitionings of $E(G)$ satisfying the conclusions of \cref{lemma-almost-factorization-technical} with $\alpha = 1/10$, $c$, and $t$ an odd integer in $[d^{c/100}, d^{c/10}]$. For any such partitioning $E(G)=E_1 \cup \dots \cup E_t$, the argument in the proof of \cref{main:pseudorandom} provides a decomposition $E(G)=E(R_1)\cup\dots\cup E(R_t)$. Recall that for all $i\in[t]$, $R_i:=W_i\cup G'_i$, where $W_i$ is an $(\alpha,r,n/2)$-good graph with $r\geq \lfloor\frac{d}{t}\left(1-\frac{16}{t^{1/3}}\right) \rfloor-1$ and $E(W_i)\subseteq E_i$, and $G'_i$ is an $r'_i$-regular graph with $r'_i \leq 40d/t^{4/3}$. In particular, by \cref{rmk:counting-completion}, each $R_i$ has at least $\left((1-\epsilon_2)\frac{d}{te^{2}}\right)^{nd/2t}$ $1$-factorizations. It follows that the multiset of $1$-factorizations of $G$ obtained in this manner has size at least $\left((1-\epsilon_3)\frac{d}{2e^2}\right)^{nd/2}$.     

To conclude the proof, it suffices to show that no $1$-factorization $\mathcal{F}=\{F_1,\dots,F_d\}$ has been counted more than $\left(1+\epsilon_4\right)^{nd/2}$ times above. Let us call an edge partitioning $E(G)=E_1\cup\dots\cup E_t$ \emph{consistent} with $\mathcal{F}$ if $E(G)=E_1\cup \dots \cup E_t$ satisfies the conclusions of \cref{lemma-almost-factorization-technical}, and $\mathcal{F}$ can be obtained from this partition by the above procedure.
It is clear that the multiplicity of $\mathcal{F}$ in the multiset is at most the number of edge partitionings consistent with $\mathcal{F}$, so that it suffices to upper bound the latter.  
For this, note that at least $d - 57d/t^{1/3}$ of the perfect matchings in $\mathcal{F}$ must have all of their edges in the same partition $E_i$. Thus, the number of edge partitionings consistent with $\mathcal{F}$ is at most ${d \choose 57d/t^{1/3}}\left(\frac{t}{2}+t^{2/3}\right)^{57nd/2t^{1/3}}d^{t}$, and observe that this last quantity can be expressed as $\left(1+\epsilon_4\right)^{nd/2}$.
\end{proof}

\section{Concluding remarks and open problems}
\label{section-concluding-rmks}

\begin{itemize}
\item In \cref{main:pseudorandom}, we proved that every $(n,d,\lambda)$-graph contains a $1$-factorization, assuming that $\lambda \leq d^{1-\varepsilon}$ and $d_0\leq d\leq n-1$ for $d_0$ sufficiently large. As we mentioned after the statement, it seems reasonable that one could, by following our proof scheme with a bit more care, obtain a bound of the form $\lambda \leq d/\log^cn$. In \cite{KS}, Krivelevich and Sudakov showed that if $d-\lambda\geq 2$ (and $n$ is even) then every $(n,d,\lambda)$-graph contains a perfect matching (and this, in turn, was further improved in \cite{CGH}). This leads us to suspect that our upper bound on $\lambda$ is anyway quite far from the truth. It will be very interesting to obtain a bound of the form $\lambda \leq d -C$, where $C$ is a constant, or even one of the form $\lambda \leq \varepsilon d$, for some small constant $\varepsilon$. Our proof definitely does not give it and new ideas are required. 
\item In \cite{KW}, Kim and Wormald showed that for every fixed $d\geq 4$, a typical $G_{n,d}$ can be decomposed into perfect matchings, such that for `many' prescribed pairs of these matchings, their union forms a Hamilton cycle (in particular, one can find a Hamilton cycle decomposition in the case that $d$ is even). Unfortunately, our technique does not provide us with any non trivial information about this kind of problem, but we believe that a similar statement should be true in $G_{n,d}$ for all $d$.
\item In \cref{main:counting}, we considered the  problem of counting the number of $1$-factorizations of a graph. We showed that the number of $1$-factorizations in $(n,d,\lambda)$-graphs is at least 
$$\left((1-o_d(1))\frac{d}{2e^2}\right)^{nd/2},$$ which is off by a factor of $2^{nd/2}$ from the conjectured upper bound (see \cite{LL}), but is still better than the previously best known lower bounds (even in the case of the complete graph) by a factor of $2^{nd/2}$. 
In an upcoming paper, together with Sudakov, we have managed to obtain an optimal asymptotic formula for the number of $1$-factorizations in $d$-regular graphs for all $d\geq \frac{n}{2}+\varepsilon n$. It is not very unlikely that by combining the techniques in this paper and the one to come, one can obtain the same bound for $(n,d,\lambda)$-graphs, assuming that $d$ is quite large (at least $\log^Cn$). It would be nice, in our opinion, to obtain such a formula for all values of $d$.
\item A natural direction would be to extend our results to the hypergraph setting. That is, let $H^k_{n,d}$ denote a $k$-uniform, $d$-regular hypergraph, chosen uniformly at random among all such hypergraphs. For which values of $d$ does a typical $H^k_{n,d}$ admit a $1$-factorization? How many such factorizations does it have? Quite embarrassingly, even in the case where $H$ is the complete $k$-uniform hypergraph, no non-trivial lower bounds on the number of $1$-factorizations are known. Unfortunately, it does not seem like our methods can directly help in the hypergraph setting. 

\end{itemize}

\end{document}